\documentclass[12pt,reqno]{amsart}
\usepackage{color,mathrsfs,epsfig,amssymb,xcolor,enumitem}

\usepackage{hyperref}\hypersetup{colorlinks=true,linkcolor=blue,citecolor=blue}

\newcommand{\N}{\mathbb{N}}
\newcommand{\R}{\mathbb{R}}

\renewcommand{\div}{{\rm div}\,}

\newcommand{\loc}{{\rm loc}}

\newcommand{\Qone}{{\mathcal{Q}}}

\newtheorem{theorem}{Theorem}
\newtheorem{lemma}[theorem]{Lemma}

\theoremstyle{definition}

\theoremstyle{remark}
\newtheorem{remark}[theorem]{Remark}

\newcommand{\RR}{\mathbb{R}}
\newcommand{\NN}{\mathbb{N}}

\begin{document}
\title[Loss of regularity for the
continuity equation]{Loss of regularity for the
continuity equation \\ with non-Lipschitz velocity field}

\author{Giovanni Alberti}
\address{Giovanni Alberti: Dipartimento di Matematica,
Universit\`a di Pisa, largo Pontecorvo 5,
56127 Pisa, Italy}
\email{giovanni.alberti@unipi.it}

\author{Gianluca Crippa}
\address{Gianluca Crippa: Departement Mathematik und Informatik, Universit\"at Basel,
Spiegelgasse 1, CH-4051, Basel, Switzerland}
\email{gianluca.crippa@unibas.ch}

\author{Anna L.~Mazzucato}
\address{Anna L.~Mazzucato: Department of Mathematics,
Penn State University, 
University Park, PA 16802, USA}
\email{alm24@psu.edu}

\begin{abstract}
We consider the Cauchy problem for the continuity equation  in space\break dimension ${d \geq 2}$. We construct a divergence-free velocity field uniformly  bounded in all Sobolev spaces $W^{1,p}$, for $1 \leq p<\infty$, and a smooth compactly supported initial datum
such that the unique solution to the continuity equation with this initial datum 
and advecting field does not belong to any Sobolev space of positive 
fractional order at any positive time.
We also construct velocity fields in 
$W^{r,p}$, with $r>1$, and solutions of the continuity equation with these 
velocities that exhibit some loss of regularity, as long as the Sobolev 
space $W^{r,p}$ does not embed in the space of Lipschitz functions.
Our constructions are based on examples of optimal mixers from the companion paper  {\em Exponential self-similar mixing by incompressible flows} (J. Amer. Math. Soc., DOI:https://doi.org/10.1090/jams/913), and have been announced in {\em Exponential self-similar mixing and loss of regularity for continuity equations} (C.~R.~Math.~Acad.~Sci.~Paris, {\bf 352} (2014), no.~11).
\end{abstract}



\maketitle

\section{Introduction}

We consider  the Cauchy problem for the continuity
equation in $d$ space dimensions:
\begin{equation}\label{e:continuity}
\left\{ \begin{array}{l}
\partial_t \rho + \div (u \rho) = 0 \\
\rho(0,\cdot) = \bar \rho \,,
\end{array}\right.
\end{equation}
where  the velocity field $u = u(t,x) : \R_+ \times \R^d \to \R^d$ and the
datum ${\bar \rho = \bar \rho(x) : \R^d \to \R}$ are given, and  $\rho =
\rho(t,x) : \R_+ \times \R^d \to \R$. Such an equation arises in different contexts and it models conservation of the scalar quantity  $\rho$ under the dynamics generated by the velocity field $u$. Often, as in the case of incompressible fluids, the (distributional) spatial divergence $\div u$ vanishes, and it is well known that, among other factors,  the behavior 
of solutions of  \eqref{e:continuity} depends on controlling the growth of the divergence of $u$. In this work we restrict to velocity fields that are divergence free,  in which case the continuity equation coincides with the transport
equation 
\begin{equation}\label{e:transport}
\partial_t \rho + u \cdot \nabla \rho = 0.
\end{equation}
Further assumptions on the velocity field $u$ will be specified  later in the paper.

The focus of this work is to investigate the regularity of the solution $\rho$ at time $t$, given a certain regularity on the initial datum $\bar\rho$ and conditions on the velocity, such a uniform-in-time bound on norms of $u$ and its derivatives. In particular, we are interested in determining what type of loss of regularity, if any, occurs in this setting. 

The question whether initial regularity is propagated for solutions of 
equation \eqref{e:continuity} has a classical and simple answer when the 
velocity field is
regular enough, namely Lipschitz continuous with respect to space uniformly with
respect to time:
\begin{equation}
\label{e:lipschitz}
| u(t,x) - u(t,y) | \leq L |x-y|,
\end{equation}
for any $t \geq 0$ and $x$, $y \in \R^d$, for some constant $L>0$ depending on
$u$, but not on $t,\,x,\,y$. The infimum over all $L$ satisfying the above
estimate is called the Lipschitz constant of $u$ and denoted by Lip$(u)$.

With this regularity on the velocity field,  the classical Cauchy-Lipschitz theory
applies, and  the Cauchy problem for the continuity
equation~\eqref{e:continuity} has a unique solution, which is moreover
transported by the unique flow of the velocity field $u$, that is,
\begin{equation}
\label{e:transported}
\rho (t,x) = \bar \rho \big( X(t,\cdot)^{-1} (x) \big) \,,
\end{equation}
where the flow $X = X(t,x) : \R_+ \times \R^d \to \R^d$ is the solution of
\begin{equation}
\label{e:ode}
\begin{cases}
\dot X(t,x) = u(t,X(t,x)) \\
X(0,x) = x \,.
\end{cases}
\end{equation}
Therefore, the solution $\rho$ at time $t$ has the same regularity as the initial datum $\bar\rho$ provided the flow $X$ is regular enough. In particular, Gr\"onwall's  estimate ensures that the flow and its inverse are Lipschitz
continuous with respect to space with a bound on their  Lipschitz constant that
is exponential in time and in~$L$, where $L$ is as the 
constant appearing in~\eqref{e:lipschitz}:
\begin{equation}
\label{e:gronwall}
{\rm Lip}\, \big(X(t,\cdot)\big) \leq \exp(tL) \,, \qquad
{\rm Lip}\, \big(X(t,\cdot)^{-1}\big) \leq \exp(tL).
\end{equation}
 In particular, from~\eqref{e:transported} it follows that
\begin{equation}
\label{e:propagated}
{\rm Lip}\, \big( \rho(t,\cdot) \big) \leq {\rm Lip}\, (\bar \rho) \, \exp(tL)
\,,
\end{equation}
and an analogous result holds for the H\"older regularity of the solution.
Alternatively, the exponential bound in~\eqref{e:propagated} can be proved by 
performing simple energy
estimates directly on   the continuity equation~\eqref{e:continuity}.


In many situations, for example when $u$ is a solution of the Navier-Stokes and Euler equations of fluid mechanics,  the velocity field is not Lipschitz continuous, but still retains some regularity in terms of weak or distributional derivatives.
In this setting, a theory of well posedness in Lebesgue spaces $L^p$, $1\leq p \leq \infty$, for the continuity
equation~\eqref{e:continuity} and the ordinary differential
equation~\eqref{e:ode} has been developed, starting with the seminal works by
DiPerna-Lions~\cite{DL} and Ambrosio~\cite{A}. In these two papers, the velocity
field is assumed to have, respectively, Sobolev or $BV$ regularity with respect
to the spatial variables, with an integral dependence on time.
(By $BV$ regularity we mean that the velocity is of bounded variation, and by Sobolev regularity we mean that derivatives of $u$, of possibly fractional order, have a certain integrability in space.)
Then one can prove existence, uniqueness, and stability of weak
solutions of the partial differential equation (PDE)~\eqref{e:continuity}, and of suitable flows
(the so-called {\em regular Lagrangian flows\/}), solutions of the ordinary
differential equation (ODE)~\eqref{e:ode}. (See~\cite{notes} for a
recent survey on research in this area.)

This well-posedness result makes it rather natural to ask whether any propagation
of regularity for solutions of~\eqref{e:continuity} holds when the advecting 
field has only the Sobolev or $BV$ regularity assumed in~\cite{A,DL}. A positive answer to 
this question would be very relevant in particular for applications to 
non-linear problems, as propagation of regularity would imply strong 
compactness of solutions, allowing the use of approximation schemes for 
non-linear problems that include continuity equations. This is  a rather delicate question 
because the flow map may be discontinuous.


To the best of our knowledge, the literature addressing propagation of regularity in the framework described above  is rather
limited. Several results of classical flavor for ``almost Lipschitz'' velocity fields (i.e., with a
modulus of continuity satisfying the Osgood condition) are collected for
instance in~\cite{chemin}. Whether or not any propagation of regularity holds for less regular fields again seems  a rather subtle question that, in particular, will depend on dimension. For instance, in one space dimension it is not difficult  to see that $BV$ regularity of the initial datum is propagated in time for solutions 
of the transport equation with a Sobolev velocity field, without additional
conditions on its divergence. By contrast,  in~\cite{colrausch} the authors show that, in two space dimensions, neither 
$BV$ regularity nor continuity of the data are preserved under the flow,
if the  velocity field lies in the intersection of a Sobolev and a H\"older space.
Their construction is based on the lack of pointwise uniqueness for solutions
of~\eqref{e:ode}, which can be used to generate trajectories that collapse quickly  by  time reversion.
See also~\cite{Figalli} for a related analysis concerning $H^{3/2}$ velocity
fields on the circle. 

The main result of this work implies that,  in two or more space dimensions, the Sobolev regularity of initial datum is not necessarily preserved in time when the velocity is not Lipschitz continuous. The loss of regularity is particularly 
striking when the velocity field has only one derivative in $L^p$, for $1\leq p<\infty$.
In this case, it is possible for smooth  initial data to loose all Sobolev
regularity, even of fractional order,  instantaneously at $t=0_+$. One should contrast this result with the fact that, within the Ambrosio-DiPerna-Lions theory, the Lebesgue $L^p$ norms of the initial datum are all exactly conserved in time, if $u$ is divergence free. The divergence-free constraint on the velocity field is a strong constraint as it implies the flow of $u$ preserves the Lebesgue measure, so blow up of the norm by concentration is more difficult or impossible to achieve. We do not expect that allowing $u$ to have a non-zero, but bounded, divergence would significantly simplify our construction. Allowing the divergence to become singular would likely lead to simpler examples of loss of regularity, but would also bring the result outside of the Ambrosio-DiPerna-Lions theory for solutions to \eqref{e:continuity}.  

This negative result on propagation of regularity in this work complements a few  positive results that have become available recently in the literature. There, (low) regularity is measured in a non-standard way, while we choose to work with traditional Sobolev spaces. The  quantitative
estimates for regular Lagrangian flows in~\cite{estimates} show that any regular
Lagrangian flow can be approximated (in the Lusin sense) with Lipschitz maps,
with a quantitative control on the Lipschitz constant of the approximation that depends 
exponentially on the size of the neglected set. This 
approximation result implies the propagation of a suitable ``Lipexp regularity''
of the initial datum (see~\cite{estimates}).  More recently, in~\cite{jabin2}  preservation  of ``a logarithm of a derivative'' for the solution of the continuity equation is established  and used  for a new theory of
existence of solutions to the compressible Navier-Stokes equations. This result
is obtained by means of a delicate argument, which utilizes the propagation of a
weighted norm of the solution, with a suitable choice of weights, via an adjoint
problem, and does not require the divergence of the velocity to be bounded.
See also~\cite{leger} for a related approach in the context 
of incompressible velocity fields, described in terms of Fourier
multipliers and proved with the use of multilinear harmonic analysis~\cite{SSS},
and the subsequent  sharp results in~\cite{hung}.

Our main result is contained in the following theorem, where $d$ again denotes the spatial dimension.

\begin{theorem}\label{t:main}
Let $d\geq 2$. There exist a bounded velocity field $v:\,[0,\infty)\times \RR^d\to \RR^d$ and a bounded 
solution $\theta:\, \,[0,\infty)\times \RR^d\to \RR$
of the continuity equation \eqref{e:continuity} with velocity $v$, 
such that:
\begin{enumerate}[label=(\roman*),ref=(\roman*)]
\item for any $1 \leq p<\infty$ the velocity field $v$ is bounded in 
$\dot{W}^{1,p}(\R^d)$ uniformly in time;
\item the initial datum $\bar \theta:=\theta(0,\cdot)$ belongs to 
$C^\infty_c(\R^d)$;
\item the solution $\theta(t,\cdot)$ does not belong to
$\dot{H}^s(\R^d)$ for any $s>0$ and $t>0$.
\end{enumerate}
Moreover both $v$ and $\theta$ have support in a fixed compact subset of $\RR^d$ for all $t\geq 0$, and
can be taken to be smooth on the complement of a point in $\R^d$.
\end{theorem}

The main strategy in the proof of Theorem \ref{t:main} above is to obtain sustained (exponential) growth of Sobolev norms with positive exponent from decay to zero of Sobolev norms with negative exponent by means of an interpolation inequality and conservation of the $L^p$ norms for solutions of \eqref{e:continuity}. The decay of negative Sobolev norms is, in turn, a consequence of results on optimal mixing, obtained by the authors \cite{ACM2}. The other main ingredient in the proof is a speed up of the growth of the Sobolev norm that eventually leads to instantaneous blow up, achieved via a suitable rescaling argument. We discuss  these ideas in more details later in the Introduction.

Theorem \ref{t:main} is stated in term of homogeneous Sobolev 
spaces $\dot{W}^{s,p}$ and $\dot{H}^s\equiv\dot{W}^{s,2}$ (defined below in Appendix~\ref{s:prelim}), given that we employ rescaling in the proof  and  homogeneous Sobolev norm have the sharpest behavior under rescaling as opposed to non-homogeneous spaces. In our results, it is equivalent to work with homogeneous or inhomogeneous norms, since in 
the example we construct both the velocity field and the solution  are uniformly bounded.
Hence the inhomogeneous norm will be bounded if and only if the homogeneous 
norm is bounded. Additionally, we state the loss of regularity for the solution 
only in terms of $L^2$-based Sobolev spaces $H^s$, $s\in \R_+$. While these 
norms have physical meaning, using $L^2$-based Sobolev norms is a technical 
restriction, which allows us to use equivalent Gagliardo-type norms for the 
fractional spaces $H^s$, $0<s<1$. These norms in turn imply an almost 
orthogonality property (see Lemma \ref{l:orto}) in these spaces. We expect that 
a statement analogous to that of Theorem \ref{t:main} can be made concerning 
blow-up of $L^p$-Sobolev norms, $1<p<\infty$.

\begin{remark} \label{r:ODE}
We work directly with the continuity equation, rather than 
with the flow induced by the advecting velocity, that is, we employ directly 
the PDE and we do not explicitly refer to the underlying ODE for the flow.  We can, however, draw the following conclusions regarding the flow of the velocity $u$.
The dramatic loss of regularity in Theorem \ref{t:main} indicates a severe lack of 
continuity of the
solution map in Sobolev spaces. Indeed, the inverse flow map $X^{-1}(t,\cdot)$ 
cannot be in any $\dot H^s (\R^d)$, with $s>0$, for $t>0$, otherwise the solution would retain 
Sobolev regularity from formula \eqref{e:transported} as the initial data is 
smooth. The representation formula \eqref{e:transported} is still valid, though 
not pointwise, for the unique weak and renormalized solution to 
\eqref{e:continuity}. In a recent paper \cite{jabin1}, Jabin gives in fact an 
example of a 
velocity field in the Sobolev space $W^{1,p}(\R^d)$, for $1\leq p < \infty$, with respect 
to the space variables, for which the associated regular
Lagrangian flow is not in $W^{1,\tilde p}(\R^d)$ for any $\tilde p$. The proof exploits a randomization 
argument in the choice of the rotations of certain basic blocks. In fact,
Jabin observed that his construct gives  the same lack of
regularity for the flow in case of velocity fields in $W^{r,p}(\R^d)$, provided $p<\frac{d}{r-1}$,
where $r$ is even allow to be larger than $1$ and $p$ to be smaller than $1$ (cf. Theorem~\ref{t:small} below). 
\end{remark}

\begin{remark}\label{r:hung}
Recently, it was shown in~\cite{hung} that adapting the proof of Theorem~\ref{t:main} and exploiting  suitable
logarithmic interpolation inequalities  leads to
examples of loss of logarithmic regularity for the solution of
the continuity equation. 
\end{remark} 

The proof of the main result in this work is based on a previous construction~\cite{ACM2} (see also~\cite{ACM1}) of optimal mixers in  the context of two-dimensional incompressible flows under Sobolev bounds on 
the velocity. Informally, mixing of $\rho$ amounts to 
creation of stretching and filamentation that accounts for 
a small negative Sobolev norm of the solution. Since the $L^2$ norm of the solution 
is preserved under the flow, by interpolation positive Sobolev norms have to 
be large. The strategy of proof has two main steps:
\begin{enumerate} [leftmargin=*,label={\bf Step \arabic*.},ref={Step 
\arabic*}]
\item Construction of a basic element that saturates the Gr\"onwall 
estimate~\eqref{e:gronwall}--\eqref{e:propagated} in integral sense: positive Sobolev norms 
of the solution grow exponentially in time. This construction is essentially 
taken from our previous paper~\cite{ACM2}. \label{i:step1}
\item Iteration and scaling, which turn the exponential growth into an 
instantaneous blow up for the positive Sobolev norms of the solution, 
while still keeping a control on the regularity of the velocity field and the 
initial datum. \label{i:step2}
\end{enumerate}
We notice that it is straightforward to saturate the Gr\"onwall estimate for \ref{i:step1} above in pointwise 
sense. The relevance of our example is, instead, to saturate it in 
integral sense.

As a matter of fact, one can prove a complementary result to that in Theorem 
\ref{t:main} in the setting of velocity fields with higher regularity, showing 
loss of some regularity on the solution. By higher regularity, we mean here  that the velocity field belongs uniformly in time to  
a Sobolev space $W^{r,p}$ with $r>1$, as long as this space does not embed in the 
Lipschitz class. 

\begin{theorem}\label{t:small}
Assume $d\geq 2$. Let $r > 1$ and $1 \leq p < \infty$ be such that $p< \frac{d}{r-1}$.
Let $T>0$ and $\sigma>0$ be fixed.
Then there exist a bounded velocity field $v$ and a solution $\theta$ 
of the continuity equation \eqref{e:continuity} with velocity $v$, both defined 
for $0 \leq t \leq T$,  such 
that:
\begin{enumerate}[label=(\roman*),ref=(\roman*)]
\item the velocity field $v$ is bounded in $\dot{W}^{r,p}(\R^d)$ uniformly 
in time;
\item the initial datum $\bar \theta:=\theta(0,\cdot)$ belongs to 
$\dot{H}^\sigma(\R^d)$;
\item the solution $\theta(t,\cdot)$ does not belong to
$\dot{H}^\sigma(\R^d)$ for any $t>0$;
\item both $v$ and $\theta$ are compactly supported in space, and
can be taken smooth on the complement of a point in $\R^d$.
\end{enumerate}
Furthermore, there is $\bar\mu = \bar\mu(r,p,d)<1$ such that
\begin{enumerate}[label=(\roman*),ref=(\roman*),start=4]
\item[(v)] for any $\mu>\bar \mu$ the solution $\theta(T,\cdot)$ does not belong to
$\dot{H}^{\mu\sigma}(\R^d)$ at the final time $T$.
\end{enumerate}
\end{theorem}

We make a few comments elucidating the similarities and differences between the two theorems:
\begin{enumerate}[leftmargin=*,label=(\arabic*), ref=(\arabic*)] 
\item The loss of regularity in Theorem~\ref{t:small} 
is much weaker than that in Theorem~\ref{t:main}. This appears as an artifact of our strategy for the proof. In fact,
the examples of optimal mixers in~\cite{ACM2}, the building blocks in the current proof,  give exponential decay of negative norms (we refer to this situation as ``exponential mixing''), if a uniform  $L^p$-bound in time is imposed only on the gradient of the velocity. If the bound is on higher derivatives, then the mixing is polynomial in time, which is not sufficient to show instantaneous blow up of the positive norms by rescaling. One can attempt to remedy this problem   in two ways and obtain a stronger version  of Theorem~\ref{t:small}, by constructing exponential mixers with uniform bounds on higher derivatives or by constructing flows that give directly growth of positive norms for the solution. Very recently, Elgindi and Zlatos have constructed more general exponential mixers under uniform bounds in $W^{s,p}$ for $1<s<2$ and $p$ close to $2$ \cite{ElgindiZlatos2018}. (See Theorem~\ref{t:mixing} for more details 
on the examples of optimal mixers from~\cite{ACM2}.)

\item  If $r<1$ and we assume the velocity field is in 
$W^{r,p}$, then in general solutions are not unique. However, we can prove the 
existence of one solution as in Theorem~\ref{t:main}. 

\item  The solution in Theorem~\ref{t:small} is unique in $L^2$  if $p 
\geq 2$. If $\sigma>d/2$,  the solution can be taken to be bounded, and is the 
unique one for any $p \geq 1$. 

\item In the proof of both theorems we do not discuss the case $p=1$ for the velocity. This 
choice is by convenience, since we define Sobolev 
spaces  using Fourier multipliers and this definition is not equivalent to the 
definition using weak derivatives and interpolation for $p=1$ and $p=\infty$. 
Clearly, the result holds also for $p=1$ as stated by embedding, as $W^{r,p}\subset 
W^{r,1}$ for any $p>1$.
\end{enumerate}

\begin{remark}\label{r:theorem3}
Further variants of Theorem \ref{t:main} are possible. For example, by 
modifying the proof of Theorem \ref{t:small}, one can prove instantaneous loss 
of some regularity, but with further constraints on the regularity of the 
velocity field and of the initial data. Informally, the regularity index $r>1$ 
must be sufficiently close to $1$, $\sigma$ must be sufficiently large, and $s$ 
sufficiently close to $\sigma$. Furthermore, these conditions are coupled 
together.  For sake of clarity and to keep the article contained, we opt not to 
include a precise statement and a proof of these variants.
\end{remark}

The paper is organized as follows. We discuss the example of optimal mixing flows from \cite{ACM2} that constitute the building blocks for the constructive proofs of our main results in Section \ref{s:mixing}. Section \ref{s:proofs} contains the proof of Theorems \ref{t:main} and \ref{t:small}. In Appendix \ref{s:prelim}, we briefly recall main facts about fractional Sobolev spaces and present an almost orthogonality result for functions in these spaces with well separated supports, which yields a key lemma for the iteration step in the proof of the main theorem. 

\subsection*{Acknowledgments}
The authors thank the anonymous referees for a careful reading of the manuscript and insightful comments. 
This work was started during a visit of the first
and third authors at the University of Basel.
Their stay has been partially supported by the
Swiss National Science Foundation grants 140232 and 156112.
The visits of the second author to Pisa have been supported by the University of Pisa PRA
project ``Metodi variazionali per problemi geometrici [Variational Methods
for Geometric Problems]''.
The second author was partially supported by the
ERC Starting Grant 676675 FLIRT and the third author by the
US~National Science Foundation grants DMS-1312727 and DMS-1615457.

\section{An example of exponential mixing} \label{s:mixing}

As outlined in the Introduction, the first step in proving our main result, 
Theorem \ref{t:main}, is the construction of a regular velocity field $u$ and 
of a regular solution $\rho$ to the continuity equation with velocity $u$ such 
that $u$ is Lipschitz uniformly in time and the positive Sobolev norms of 
$\rho$ grow exponentially in time. 

Given that all Lebesgue norms of $\rho$ are constant in time,  one way to 
achieve the exponential growth in time of the positive Sobolev norms is to 
ensure an exponential decay of the negative norms in view of the interpolation 
inequality \eqref{e:interpolation}. Negative Sobolev norms will decay in time 
if the flow of $u$ is sufficiently ``mixing'', which informally means that the 
action of the flow generates small scales in the solution. It was recently 
shown \cite{ikx,seis} that, under a $W^{1,p}$ bound on $u$ that is uniform in time, 
where $1<p\leq \infty$, the rate of decay of negative Sobolev norms for 
the solution $\rho$  is indeed at most exponential in time. 

In \cite{ACM2}, we constructed examples of velocity fields satisfying the 
required $W^{1,p}$ bound and solutions to the associated continuity equation that 
saturate the exponential rate of decay of the negative norms (see 
\cite{ElgindiZlatos2018,zlatos} for related examples). We will refer to these examples as 
``optimal mixers''. We also consider velocity fields with higher regularity, 
see~\eqref{e:expfield} and \eqref{e:expmix} below.
Our examples are based on a quasi self-similar geometric construction.

\begin{theorem}[\protect\cite{ACM2}]\label{t:mixing}
Assume $d\geq 2$ and let $\Qone \subset \R^d$ be the open cube with unit side 
centered at the origin of $\R^d$. There exist a velocity field $u \in C^\infty ( [0,+\infty[ \times \R^d)$ and
a corresponding (non trivial) solution $\rho \in C^\infty ( [0,+\infty[ \times 
\R^d)$ of the continuity equation \eqref{e:continuity} such that:
\begin{enumerate}[label=(\roman*),ref=(\roman*)]
\item $u(t,\cdot)$ is bounded, divergence-free, and compactly supported in 
$\Qone$ for any $t$;
\item $\rho(t,\cdot)$ has zero average and is bounded and compactly 
supported in $\Qone$ for any $t$;
\item $u(t,\cdot)$ belongs to ${\rm Lip}\,(\R^d)$ uniformly in time;
\item for any $r \geq 0$ and $1 \leq p \leq \infty$, there exist constants 
$b>0$ and $B_r > 0$ such that
\begin{equation}\label{e:expfield}
\| u(t,\cdot) \|_{\dot{W}^{r,p}(\R^d)} \leq B_r \exp \big( (r-1) \, b t \big)\,;
\end{equation}
\item for any $0 < s < 2$, there exist constants $c>0$ and $\hat{C}_s >0$ 
such that
\begin{equation}\label{e:expmix}
\| \rho(t,\cdot) \|_{\dot{H}^{-s}(\R^d)} \leq \hat{C}_s \exp(- s c t)
\qquad t \geq 0 \,.
\end{equation}
\end{enumerate}
\end{theorem}

\begin{remark}\label{r:gronwall}
By the interpolation formula \eqref{e:interpolation} with $s_1 = -s$, $s_2=s$, 
and $\vartheta = 1/2$, which reads
$$
\| \rho(t,\cdot) \|_{L^2(\R^d)}
\leq
\| \rho(t,\cdot) \|^{1/2}_{\dot{H}^{-s}(\R^d)} \| \rho(t,\cdot) 
\|^{1/2}_{\dot{H}^{s}(\R^d)} \,,
$$
together with the fact that the $L^2$ norm of the solution $\rho$ is conserved 
in time, estimate \eqref{e:expmix} implies the following lower bound:
\begin{equation}\label{e:expsol}
\begin{aligned}
\| \rho(t,\cdot) \|_{\dot{H}^{s}(\R^d)}
& \geq
\| \rho(t,\cdot) \|^2_{L^2(\R^d)} \| \rho(t,\cdot) \|^{-1}_{\dot{H}^{-s}(\R^d)}
\geq
\| \rho(0,\cdot) \|_{L^2(\R^d)}^2 \hat{C}_s^{-1} \exp(s c  t) \\
& = C_s \exp(s c  t)
\quad \text{ for every $0<s<2$.}
\end{aligned}
\end{equation}
This bound can be seen as saturating  Gr\"onwall 
inequality~\eqref{e:gronwall}--\eqref{e:propagated} and will be used in the 
proof of Theorems~\ref{t:main} and~\ref{t:small}.

Estimate \eqref{e:expsol} is  weaker than  \eqref{e:expmix}, in the sense that growth of the positive norms does not imply decay of the negative norms, yet it is still enough to establish the loss of regularity. In fact, what is needed is a lower bound that implies growth of positive norms faster than polynomial, but not necessarily exponential in time.  Furthermore, blow up is a {\em local} phenomenon, while mixing is inherently a {\em non-local} property. One can then ask whether it is possible to use a simpler flow as building block for the construction in Section \ref{s:proofs}. Intuitively, a  circular-type flow, where the angular velocity is suitably chosen, could suffice. However, a preliminary analysis indicates that making this construction rigorous is rather delicate and we choose instead to exploit the work we have already done for the optimal mixers. We reserve to investigate this point further in future work.   

\end{remark}

\begin{remark} The example in~\cite{ACM2} was constructed in two space 
dimensions. However, it is not difficult to modify such construction to make it 
$d$ dimensional. One can argue as follows. Call $\tilde u(t,x_1,x_2)$ and 
$\tilde \rho(t,x_1,x_2)$ the velocity field and solution constructed 
in~\cite{ACM2}. Fix $\eta = \eta(x_3,\ldots,x_d)$ smooth, with compact support 
in $B(0,1/4) \subset \R^{d-2}$ and $\eta=1$ on $B(0,1/8)$. Moreover, fix 
$\bar\eta = \bar\eta(x_3,\ldots,x_d)$ smooth, with compact support in $B(0,1/2) 
\subset \R^{d-2}$ and $\bar\eta=1$ on $B(0,1/4)$. Then, setting
$$
u(t,x_1,\ldots,x_d) = \tilde u(t,x_1,x_2) \bar\eta(x_3,\ldots,x_d) \,, \quad
\rho(t,x_1,\ldots,x_d) = \tilde\rho(t,x_1,x_2) \eta(x_3,\ldots,x_d),
$$
yields a velocity field $u$ and solution $\rho$ defined on $\R^d$ that satisfy 
the conclusions of Theorem~\ref{t:mixing}.
\end{remark}

\section{Proof of Theorems~\ref{t:main} and \ref{t:small}} \label{s:proofs}

The main idea in the proof of both Theorems~\ref{t:main} and \ref{t:small} consists in 
constructing the velocity field $v$ and the associated solution $\theta$ by 
patching together a countable number of velocity fields $u_n$ and solutions 
$\rho_n$ saturating Gr\"onwall's inequality, the 
existence of which follows from Theorem~\ref{t:mixing} (see also
Remark~\ref{r:gronwall}),  after rescaling in both space and time, and in size 
as well for the solution, in a suitable manner. We then translate them so as 
to be supported in cubes that accumulate towards a point in space. The pairs 
$u_n$, $\rho_n$ serve as building blocks in our construction.

The choice of the rescaling parameters is the key point of the proof and is done
in such a way that the regularity of the velocity field and that of 
the initial datum for the solution can be controlled, while at the same time  
any regularity of the solution at later times is destroyed.

In the proof of Theorem~\ref{t:main}, we choose the parameters for the 
rescaling to ensure that the velocity field has Sobolev regularity 
uniformly in time, while we exploit that the rescaling in time can be chosen so 
that the rate of exponential growth of the Sobolev
norms of the solution grows enough at each step to become in the
limit an instantaneous blow up.

When $r>1$, in view of estimate \eqref{e:expfield}, the
Sobolev norm of order $r$ of the velocity field is not uniformly bounded in 
time, rather it grows exponentially in time. One then needs to balance 
this exponential growth with the exponential growth for the Sobolev norms of 
the solution, which accounts for the weaker result in Theorem~\ref{t:small}.

\subsection{A geometric construction.}
We fix a sequence $\{ \lambda_n \}$, with $\lambda_n>0$ and $\lambda_n \downarrow 0$, to be determined later.
For every $n\in \NN$, we consider an open cube $Q_n$ in $\R^d$ with side of 
length $3\lambda_n$.
The first condition we impose on $\{ \lambda_n \}$ is that
\begin{equation}\label{c:A}
\tag{${\rm A}$}
\left.\begin{array}{c}
\text{$\{ Q_n\}$ can be chosen to be contained} \\
\text{in a compact set and convergent to a point}
\end{array}\right\}
\qquad \Longleftarrow \qquad
\sum_n \lambda_n < \infty\,.
\end{equation}
If \eqref{c:A} is satisfied, we can choose the centers of the cubes so that
all cubes are contained in a compact set in $\R^d$ and accumulate to a point,
i.e., every open ball centered at such a point contains all the cubes
but finitely many. By saying that condition \eqref{c:A} is satisfied, we mean 
that the condition on the right hand side, which implies the one on the left, 
is satisfied. The same will be true for conditions \eqref{c:B}--\eqref{c:D} 
(and variants) below.

Given $u$ and $\rho$ as in Theorem~\ref{t:mixing}, we fix two more sequences 
$\{ \tau_n \}$ and $\{ \gamma_n \}$ with $\tau_n>0$ and $\gamma_n>0$ 
also to be  determined later, where $\tau_n \downarrow 0$. 
Up to a translation,  in each cube $Q_n$ we set
\begin{equation}\label{e:pieces}
u_n (t,x) = \frac{\lambda_n}{\tau_n} \, u \left( \frac{t}{\tau_n} , \frac{x}{\lambda_n} \right) \,,
\qquad
\rho_n(t,x) = \gamma_n \, \rho \left( \frac{t}{\tau_n} , \frac{x}{\lambda_n} \right) \,.
\end{equation}
It is immediate to check that $\rho_n$ is a solution of the continuity equation
with velocity field $u_n$, that is,
\begin{equation}\label{e:eqn}
\partial_t \rho_n + \div ( u_n \rho_n) = 0
\qquad \text{on $\R^d$.}
\end{equation}
Moreover, we observe that for every $n$ we have
$$
{\rm dist}\, (\text{Supp}\, \rho_n \,,\, Q_n^c) \geq \lambda_n \,,
$$
since the support of $\rho_n$ is contained in a square of side $\lambda_n$  with the same center as
$Q_n$, the side of which has length $3\lambda_n$. 

We set
\begin{equation}\label{e:sums}
v := \sum_n u_n \,,
\qquad
\theta := \sum_n \rho_n \,,
\end{equation}
where convergence of both series is taken pointwise almost everywhere.
We will prove later that, for a suitable choice of the three sequences of 
parameters, $\lambda_n$, $\tau_n$, $\gamma_n$, the function 
$\theta$ is a (weak) solution of the continuity equation
with velocity field $v$: 
\begin{equation}\label{e:sumsol}
\partial_t \theta + \div (v \theta) = 0 \,.
\end{equation}

\subsection{Regularity of the velocity field}
We estimate the $\dot{W}^{r,p}(\R^d)$ norm of the velocity
field $v$ at a fixed time $t$:
$$
\begin{aligned}
\| v(t,\cdot) \|_{\dot{W}^{r,p}(\R^d)}
& \stackrel{\eqref{e:sums}}{\leq}
\sum_n \| u_n(t,\cdot) \|_{\dot{W}^{r,p}(\R^d)}
\stackrel{\eqref{e:pieces}}{=}
\sum_n \frac{\lambda_n}{\tau_n}
\left\| u \left( \frac{t}{\tau_n} , \frac{\cdot}{\lambda_n} \right) \right\|_{\dot{W}^{r,p}(\R^d)} \\
& \stackrel{\eqref{e:scaleW}}{=}
\sum_n \frac{\lambda_n}{\tau_n} \lambda_n^{\frac{d}{p}-r}
\left\| u \left( \frac{t}{\tau_n} , \cdot \right) \right\|_{\dot{W}^{r,p}(\R^d)} \\
& \stackrel{\eqref{e:expfield}}{\leq}
\sum_n \frac{\lambda_n^{1-r+\frac{d}{p}}}{\tau_n} B_r
\exp \left( \frac{ (r-1)bt}{\tau_n} \right) \,,
\end{aligned}
$$
where we have explicitly referenced which property we are using for each 
inequality. Therefore, the fact that $v(t,\cdot)$ belongs to $\dot{W}^{r,p}(\R^d)$ is 
implied by the following condition:
\begin{equation}\label{c:B}
\tag{${\rm B}$}
v(t,\cdot) \in \dot{W}^{r,p}(\R^d)
\qquad \Longleftarrow \qquad
\sum_n \frac{\lambda_n^{1-r+\frac{d}{p}}}{\tau_n}
\exp \left( \frac{ (r-1)bt}{\tau_n} \right) < \infty \,.
\end{equation}
Moreover, the definition of $u_n$, which is given in \eqref{e:pieces}, and the 
fact that the
cubes $\{ Q_n \}$ have been chosen to be pairwise disjoint, gives the following 
implication:
\begin{equation}\label{c:B'}
\tag{$\tilde{\rm B}$}
v(t,\cdot) \in L^\infty(\R^d) \text{ uniformly in $t$}
\qquad \Longleftarrow \qquad
\left\{ \frac{\lambda_n}{\tau_n} \right\} \text{ bounded.}
\end{equation}

\subsection{Regularity of the initial datum}
For a given $\sigma \geq 0$, we estimate the $\dot{H}^\sigma(\R^d)$ norm of the 
initial
datum $\Bar\theta=\theta(0,\cdot)$ as follows:
$$
\begin{aligned}
\| \Bar \theta(\cdot) \|_{\dot{H}^\sigma(\R^d)}
& \stackrel{\eqref{e:sums}}{\leq}
\sum_n \| \rho_n(0,\cdot) \|_{\dot{H}^\sigma(\R^d)}
\stackrel{\eqref{e:pieces}}{=}
\sum_n \gamma_n \left\| \rho \left( 0,\frac{\cdot}{\lambda_n} \right) \right\|_{\dot{H}^\sigma(\R^d)} \\
& \stackrel{\eqref{e:scaleH}}{=}
\sum_n \gamma_n \lambda_n^{\frac{d}{2} - \sigma} \| \rho \|_{\dot{H}^\sigma(\R^d)} \,.
\end{aligned}
$$
Hence, we obtain the further condition
\begin{equation}\label{c:C}
\tag{${\rm C}$}
\Bar\theta \in \dot{H}^\sigma(\R^d)
\qquad \Longleftarrow \qquad
\sum_n \gamma_n \lambda_n^{\frac{d}{2} - \sigma} < \infty \,.
\end{equation}

In Theorem~\ref{t:main}, we require that the solution $\theta$ is bounded, 
which follows if $\Bar\theta$ is bounded. Using
the definition of $\rho_n$ in \eqref{e:pieces} and again the fact that the
cubes $\{ Q_n \}$ have been chosen to be pairwise disjoint gives the following 
implication:
\begin{equation}\label{c:C'}
\tag{$\tilde{\rm C}$}
\theta(t,\cdot) \in L^\infty(\R^d) \text{ uniformly in $t$}
\qquad \Longleftarrow \qquad
\{\gamma_n\} \text{ bounded.}
\end{equation}
In Theorem~\ref{t:small}, we only require that $\theta$ belongs to
$L^2(\R^d)$ uniformly in time, which again follows if $\bar\theta \in 
L^2(\R^d)$. This requirement corresponds to condition
\eqref{c:C} with $\sigma=0$, that is,
\begin{equation}\label{c:Chat}
\tag{$\hat{\rm C}$}
\theta(t,\cdot) \in L^2(\R^d) \text{ uniformly in $t$}
\qquad \Longleftarrow \qquad
\sum_n \gamma_n \lambda_n^{\frac{d}{2}} < \infty \,.
\end{equation}

\subsection{Loss of regularity for the solution}
To show loss of regularity for $\theta$ at any $t>0$, it is enough to show 
that the solution does not belong to $\dot H^s(\R^d)$, for  $0<s<1$, since 
$\theta$ is at least in $L^2$ (by condition \eqref{c:C'} or \eqref{c:Chat} and 
the compact support above) and $H^s (\R^d)= \dot H^s (\R^d) \cap L^2 (\R^d)$ forms a scale of 
spaces ($H^s(\RR^d)\subset H^r(\RR^d)$, $s>r$).
We then fix an arbitrary $0<s<1$, and estimate the norm of $\theta$ in $\Dot 
H^s(\R^d)$ from below as follows: 
$$
\begin{aligned}
\| \theta(t,\cdot) \|^2_{\dot{H}^s(\R^d)}
& \stackrel{\eqref{e:nonlocal}}{\geq}
\limsup_{N\to\infty} \sum_{n=1}^N \left( \| \rho_n(t,\cdot) \|^2_{\dot{H}^s(\R^d)}
-\frac{C_d}{s} \frac{1}{\lambda_n^{2s}} \| \rho_n(t,\cdot) \|^2_{L^2(\R^d)} 
\right) \\
& \stackrel{\eqref{e:pieces}}{=}
\limsup_{N\to\infty} \sum_{n=1}^N  \left( \gamma_n^2 \left\| \rho \left( \frac{t}{\tau_n} , 
\frac{\cdot}{\lambda_n} \right) \right\|^2_{\dot{H}^s(\R^d)}
 -\frac{C_d}{s} \frac{1}{\lambda_n^{2s}} \gamma_n^2
 \left\| \rho \left( \frac{t}{\tau_n} , \frac{\cdot}{\lambda_n} \right) 
\right\|^2_{L^2(\R^d)} \right)
\\
& \stackrel{\eqref{e:scaleH}}{=}
\limsup_{N\to\infty} \sum_{n=1}^N \gamma_n^2 \lambda_n^{d-2s}
\Bigg[
\left\| \rho \left( \frac{t}{\tau_n} , \cdot \right) \right\|^2_{\dot{H}^s(\R^d)}
-\frac{C_d}{s} \left\| \rho \left( \frac{t}{\tau_n} , \cdot \right) \right\|^2_{L^2(\R^d)}
\Bigg] \\
& \stackrel{\eqref{e:expsol}}{\geq}
\limsup_{N\to\infty} \sum_{n=1}^N \gamma_n^2 \lambda_n^{d-2s}
\Bigg[
C_s^2 \exp \left( \frac{2s c t}{\tau_n} \right) - \frac{C_d \hat{C}_0}{s}
\Bigg] \,,
\end{aligned}
$$
where again we have referenced the property used at each step. Here, we exploit an almost orthogonality property in $\dot{H}^s$ for functions with disjoint support, established in Lemma \ref{l:orto} in Appendix \ref{s:prelim}.   
In the last inequality,  we have denoted 
the $L^2(\R^d)$ norm of $\rho$, which is conserved in time,  by $\hat{C}_0$.
Therefore, since we require $\tau_n \downarrow 0$, for any given $t>0$ we have
\begin{equation}\label{c:D}
\tag{${\rm D}$}
\theta(t,\cdot) \not \in \dot{H}^s(\R^d)
\qquad \Longleftarrow \qquad
\sum_n \gamma_n^2 \lambda_n^{d-2s}
\exp \left( \frac{2s c t}{\tau_n} \right)
= \infty \,.
\end{equation}

\subsection{Verify that $\theta$ is a solution of the continuity equation with 
velocity $v$}
We assume that the sequences $\{\lambda_n\}$, $\{\tau_n\}$, and $\{\gamma_n\}$ 
have been chosen
so that conditions \eqref{c:B}, \eqref{c:B'}, and \eqref{c:C'} (or 
\eqref{c:Chat}) are satisfied. Then, the
series in \eqref{e:sums} converge strongly in $L^2_{\loc}(\R^d)$ on any 
finite time interval. This fact, together
with the validity of \eqref{e:eqn} and the fact that the cubes $\{ Q_n \}$ have been chosen
to be pairwise disjoint, implies the validity of \eqref{e:sumsol}.

\subsection{Proof of Theorem~\ref{t:main}} We are now in the position to 
complete the proof of our main result.

In the case $r=1$, condition~\eqref{c:B} can be restated as 
\begin{equation}\label{c:Bhat}
\tag{$\hat{\rm B}$}
v \in L^\infty((0,\infty);\dot{W}^{1,p}(\R^d))
\qquad \Longleftarrow \qquad
\sum_n \frac{\lambda_n^{\frac{d}{p}}}{\tau_n} < \infty \,.
\end{equation}
Our task is to find sequences $\{ \lambda_n\}$, $\{ \tau_n \}$, and $\{ \gamma_n \}$ such that conditions \eqref{c:A}, \eqref{c:Bhat}, \eqref{c:B'}, \eqref{c:C}, \eqref{c:C'}, and \eqref{c:D} are satisfied, for any $1<p<\infty$, $\sigma>0$, $s>0$, and $t>0$. Choosing
\begin{equation}
\label{e:choice1}
\tau_n = \frac{1}{n^3} \qquad \text{ and } \qquad \lambda_n = e^{-n}\,,
\end{equation}
we immediately see that~\eqref{c:A}, \eqref{c:Bhat}, and~\eqref{c:B'} are satisfied. If we take
\begin{equation}
\label{e:choice2}
\gamma_n = e^{-n^2},
\end{equation}
we see that~\eqref{c:C} and~\eqref{c:C'} are satisfied. We are left with 
having to check condition~\eqref{c:D}. With the choices in~\eqref{e:choice1} 
and~\eqref{e:choice2} for the parameters, the series appearing in condition 
\eqref{c:D} reads:
$$
\sum_n \exp(-2n^2) \,
\exp \big( -(d-2s) n \big) \,
\exp (2sct n^3),
$$
which clearly diverges since $s c  t >0$. This concludes the proof of 
Theorem~\ref{t:main}. \hfill \qed

\subsection{Proof of Theorem~\ref{t:small}} The proof is a variation of the previous one.
We fix $T>0$ and choose
$$
\tau_n = \frac{1}{n}\,.
$$
We want to check \eqref{c:B} for $0 \leq t \leq T$. Such a condition now reads
\begin{equation}\label{e:BT}
\sum_n n \lambda_n^{1-r+\frac{d}{p}}
\exp \big( n (r-1) b T \big) < \infty \,.
\end{equation}
By the assumption $p< \frac{d}{r-1}$ (expressing the fact that we consider a 
Sobolev space that does not embed in the Lipschitz class), 
we have that
$$
\beta = 1-r+\frac{d}{p} > 0 \,.
$$
For $\alpha>0$ to be determined, we choose
$$
\lambda_n = \exp ( - \alpha T n) \,.
$$
Then, condition \eqref{e:BT} becomes
$$
\sum_n n \exp ( - \alpha \beta T n )
\exp \big( n (r-1)bT \big) < \infty \,.
$$
This series indeed converges provided that
\begin{equation}
\label{e:choicealpha}
- \alpha \beta T + (r-1)bT < 0 \,, \quad \text{ i.e. }\quad
\alpha > \frac{(r-1)b}{\beta} \,,
\end{equation}
which is admissible. We remark that the parameters $\alpha$ and $\beta$ can be chosen independent of $T$, since $b$ is independent of $T$.  Under this condition, \eqref{c:B} holds and
also \eqref{c:A} and \eqref{c:B'} follow.

With this choice of $\lambda_n$, condition \eqref{c:C} becomes
$$
\sum_n \gamma_n \exp \left( \alpha \left( \sigma - \frac{d}{2} \right) T n
\right)< \infty \,.
$$
We can guarantee that this condition holds if
$$
\gamma_n = \frac{1}{n^2} \exp \left( \alpha \left( \frac{d}{2} -\sigma \right) T 
n \right).
$$
Moreover, condition \eqref{c:Chat} becomes
$$
\sum_n \frac{1}{n^2} \exp \left( \alpha \left( \frac{d}{2} -\sigma \right) T n \right)
\Bigg[ \exp ( - \alpha T n ) \Bigg]^{\frac{d}{2}}
=
\sum_n \frac{1}{n^2} \exp \left( - \alpha  \sigma  T n \right)
< \infty \,,
$$
and we can verify that it is satisfied. We note in passing that, in the case 
$\sigma>d/2$,
the stronger condition \eqref{c:C'} in fact holds.

We now substitute all the above expressions into \eqref{c:D} and obtain that 
this condition is equivalent to having: 
$$
\begin{aligned}
 \sum_n \left( \gamma_n \lambda_n^{\frac{d}{2} - \sigma} \right)^2 \lambda_n^{2(\sigma-s)}
 \exp \left( \frac{2s c t}{\tau_n} \right)
= & \sum_n \frac{1}{n^4}  \exp \big( - 2 (\sigma - s) \alpha T n \big)
\exp \big( 2s c t n \big) \\
= & \sum_n \frac{1}{n^4}  \exp \Big( 2 n \big(-  (\sigma - s) \alpha T  + s c t \big)  \Big)
= \infty \,.
\end{aligned}
$$
The above series diverges if
$$
-  (\sigma - s) \alpha T  + s c t > 0\,,
$$
which is (trivially) the case for any $t \geq 0$ when $s>\sigma$. In the relevant case $0<s \leq \sigma$, 
we see
that the norm of the solution in $\dot{H}^s(\R^d)$ becomes infinity only after 
a time $(\sigma - s) \alpha T / s c $.
Recalling that we control the Sobolev norm of the velocity field up to time $T$ only,
only those regularity indices $s$ that satisfy
\begin{equation}
\label{e:percent}
\frac{(\sigma - s) \alpha T}{s c }
< T
\qquad \Longleftrightarrow \qquad
(\sigma - s) \alpha
< s c  \,,
\end{equation}
lead to blow up of the norm.
Recalling~\eqref{e:choicealpha}, we can rewrite~\eqref{e:percent} as
$$
\frac{(\sigma-s)(r-1)b}{s c  \beta} < 1
\qquad \Longleftrightarrow \qquad
\frac{s}{\sigma} > 1- \frac{c\beta}{c\beta + (r-1)b} =: \bar\mu \,.
$$
We observe that, with this definition, $\bar\mu$ is independent of $T$, as again $\beta$ and $b$ are, and it also satisfies 
$0<\bar{\mu}<1$. Therefore, the amount of regularity for the solution that 
can be lost is $\bar\mu\sigma$. Theorem \ref{t:small} is now proved. \hfill \qed

\begin{remark} 
 Although $\bar\mu$ is independent of $T$ and $T$ is chosen arbitrarily, we cannot conclude from our construction that such regularity is lost instantaneously by choosing a sequence of times $T=T_n\to 0$, as both the velocity field and the initial data, hence the entire solution, depend on $T$ in a fundamental way.
\end{remark}

\appendix

\section{Fractional Sobolev spaces and almost orthogonality} \label{s:prelim}

We recall some basic definitions and 
properties of homogeneous Sobolev spaces of real order that are used extensively 
in this paper. For a complete exposition we refer the reader for instance to \cite{chemin,palatucci,grafakos,triebel}.

\medskip

Let $s \geq 0$ and $1<p<\infty$. We say that a tempered distribution $f \in \mathcal{S}'(\R^d)$
belongs to the homogeneous Sobolev space $\dot{W}^{s,p}(\R^d)$ if the Fourier transform $\hat f$
of $f$ belongs to $L^1_\loc(\R^d)$ and
\begin{equation}\label{e:defW}
\mathcal{F}^{-1} \left( |\xi|^s \hat{f}(\xi) \right)
\in L^p(\R^d) \,.
\end{equation}
If this is the case, we define $\left\| f  \right\|_{\dot{W}^{s,p}(\R^d)}$ to be the
$L^p(\R^d)$ norm of the function in \eqref{e:defW}.

If $s = k \in \N$, then $\dot{W}^{k,p}(\R^d) \cap L^p(\R^d)$ coincides
with the usual Sobolev space $W^{k,p}(\R^d)$ consisting of those $L^p(\R^d)$ functions possessing
weak derivatives of order less or equal than $k$ in $L^p(\R^d)$. This 
equivalence however does not hold in the
borderline cases $p=1$ and $p=\infty$. 
We also recall the well-known fact that $\Dot W^{r,p}(\R^d)$ continuously 
embeds into the space of Lipschitz continuous functions provided $p>\frac{d}{r-1}$.


Throughout this work, we make frequent use  of  the behavior of homogeneous Sobolev norms of a given function $f$
under rescaling . If we set
$$
f_\lambda(x) = f \left(\frac{x}{\lambda} \right) \,,
$$
then it holds
\begin{equation}\label{e:scaleW}
\left\| f_\lambda \right\|_{\dot{W}^{s,p}(\R^d)}
=
\lambda^{\frac{d}{p}-s} \left\| f  \right\|_{\dot{W}^{s,p}(\R^d)} \,.
\end{equation}


In the particular case $p=2$ we use the notation
$\dot{H}^s(\R^d)$. For $s \in \R$, a tempered distribution $f \in \mathcal{S}'(\R^d)$
belongs to the homogeneous Sobolev space $\dot{H}^s(\R^d)$ if the Fourier transform $\hat f$
of $f$ belongs to $L^1_\loc(\R^d)$ and
\begin{equation}\label{e:defH}
\| f \|^2_{\dot{H}^s(\R^d)} = \int_{\R^d} | \xi |^{2s} | \hat{f}(\xi) |^2 \, d\xi < \infty \,.
\end{equation}
As before, it holds
\begin{equation}\label{e:scaleH}
\left\| f_\lambda \right\|_{\dot{H}^s(\R^d)}
=
\lambda^{\frac{d}{2}-s} \left\| f  \right\|_{\dot{H}^s(\R^d)} \,.
\end{equation}


In fact, more general definitions can be given removing the requirement that $\hat{f} \in L^1_\loc(\R^d)$,
setting the entire theory in the context of tempered distributions modulo polynomials (as done, for instance, in~\cite{grafakos}).
In this paper, we always deal with bounded functions with compact support, for which the
local summability of the Fourier transform is guaranteed. Therefore we prefer to follow the
somehow simplified approach above.

In our work, homogeneous spaces are used only to measure the ``size'' of 
given functions and velocity fields, which in fact will be typically regular 
but with large norm. In particular, the issue of completeness (which holds only 
for specific values of $s$ given our definition) and seminorms will not arise. 
With abuse of language, we then refer to seminorms as norms.

The following interpolation property holds. If $s_1 < s < s_2$ and $s = 
\vartheta s_1 + (1-\vartheta) s_2$, then
\begin{equation}\label{e:interpolation}
\| f \|_{\dot{H}^{s}(\R^d)}
\leq
\| f \|^\vartheta_{\dot{H}^{s_1}(\R^d)}
\| f \|^{1-\vartheta}_{\dot{H}^{s_2}(\R^d)} \,.
\end{equation}
In particular $\dot{H}^{s_1}(\R^d) \cap 
\dot{H}^{s_2}(\R^d) \subset \dot{H}^{s}(\R^d)$. Analogous inequalities hold for $\dot{W}^{s,p}(\R^d)$, but 
we will not need them in this paper.

\medskip

We recall that, for every $0<s<1$, the $\dot{H}^s(\R^d)$ (semi) norm is
equivalent (up to a dimensional factor) to the Gagliardo (semi) norm, given by :
\begin{equation}\label{e:gagliardo}
\left( \int_{\R^d} \int_{\R^d} \frac{| f(x) - f(y)|^2}{|x-y|^{d+2s}} \, dx \, dy \right)^{1/2} \,.
\end{equation}
We recall that one can equivalently introduce Gagliardo (semi) norms that are 
$L^p$ based for  $1\leq p <\infty$:
\[
   \left( \int_{\R^d} \int_{\R^d} \frac{| f(x) - f(y)|^p}{|x-y|^{d+sp}} \, dx 
\, dy \right)^{1/p} \,.
\]
However, the space of $L^p$ functions for which this norm is finite is a 
Besov space, namely the space $\Dot B^s_{p,p}$, and $\Dot{W}^{s,p} = \Dot 
B^s_{p,p}$ if and only if $p=2$ (see e.g. \cite[Chapter 2]{triebel}). This is 
one of the reasons why we state the loss of regularity in $L^2$-Sobolev spaces 
only.

Using \eqref{e:gagliardo}, we prove a localization property in $\Dot H^s$ in 
the following sense. If a function has compact support, then we exhibit a 
quantitative relation between the Sobolev norm of this function on the whole 
space and the Sobolev norm of the function
localized on a set containing its support. This lemma is essentially Lemma 
5.1 in \cite{palatucci}, and we include a proof here for the reader's sake.

\begin{lemma}
Let $0<s<1$. Let $K \subset \Omega \subset \R^d$ and assume ${\rm 
dist}\,(K,\Omega^c) = \lambda>0$. Then, for every function $f \in 
\dot{H}^s(\R^d)$ with support contained in $K$, 
$$
\int_{\R^d} \int_{\R^d} \frac{| f(x) - f(y)|^2}{|x-y|^{d+2s}} \, dx \, dy
\leq
\int_{\Omega} \int_{\Omega} \frac{| f(x) - f(y)|^2}{|x-y|^{d+2s}} \, dx \, dy
+
\frac{C_d}{s} \frac{1}{\lambda^{2s}} \| f \|_{L^2(\R^d)}^2 \,,
$$
where $C_d$ is a dimensional constant.
\end{lemma}
\begin{proof}
Since $f$ vanishes in the complement of $\Omega$, we have
$$
\int_{\R^d} \int_{\R^d} \frac{| f(x) - f(y)|^2}{|x-y|^{d+2s}} \, dx \, dy
=
\int_{\Omega} \int_{\Omega} \frac{| f(x) - f(y)|^2}{|x-y|^{d+2s}} \, dx \, dy
+
2 \int_{\Omega^c} \int_{\Omega} \frac{| f(x) |^2}{|x-y|^{d+2s}} \, dx \, dy \,.
$$
Therefore we need to estimate the last integral. We can compute
$$
\begin{aligned}
\int_{\Omega^c} \int_{\Omega} \frac{| f(x) |^2}{|x-y|^{d+2s}} \, dx \, dy
& =
\int_{\Omega}  |f(x)|^2 \int_{\Omega^c}  \frac{{\mathbf 1}_K(x)}{|x-y|^{d+2s}} \, dy \, dx \\
& \leq
\int_{\Omega}  |f(x)|^2 \int_{B(x,\lambda)^c}  \frac{{\mathbf 1}_K(x)}{|x-y|^{d+2s}} \, dy \, dx \\
& =
C_d \int_{\Omega}  |f(x)|^2 \int_\lambda^\infty \frac{r^{d-1}}{r^{d+2s}} \, dr
=
\frac{C_d}{2s} \frac{1}{\lambda^{2s}} \| f \|_{L^2(\R^d)}^2  \,. 
\end{aligned}
$$
This concludes the proof.
\end{proof}

We use the previous lemma to show a sort of ``orthogonality property'' for
Sobolev functions with disjoint supports.

\begin{lemma}
\label{l:orto}
Let $0<s<1$. For $i=1,2$, let $K_i \subset \Omega_i \subset \R^d$ and assume ${\rm dist}\,(K_i,\Omega_i^c) = \lambda_i > 0$. Moreover, assume  $\Omega_1 \cap \Omega_2 = \emptyset$. Then, given
$f_i \in \dot{H}^s(\R^d)$ with support contained in $K_i$, there exists a 
dimensional constant $C_d$ such that:
$$
\| f_1 + f_2 \|^2_{\dot{H}^s(\R^d)}
\geq
\| f_1 \|^2_{\dot{H}^s(\R^d)}
+
\| f_2 \|^2_{\dot{H}^s(\R^d)}
- \frac{C_d}{s}
\Bigg[ \frac{1}{\lambda_1^{2s}} \| f_1 \|_{L^2(\R^d)}^2 +
\frac{1}{\lambda_2^{2s}} \| f_2 \|_{L^2(\R^d)}^2 \Bigg] \,.
$$
The above formula generalizes to:
\begin{equation}\label{e:nonlocal}
\left\| \sum_n f_n \right\|^2_{\dot{H}^s(\R^d)}
\geq
\limsup_{N \to \infty}
\sum_{n=1}^N \left[
\| f_n \|^2_{\dot{H}^s(\R^d)}
- \frac{C_d}{s}
 \frac{1}{\lambda_n^{2s}} \| f_n \|_{L^2(\R^d)}^2  \right] \,.
\end{equation}
\end{lemma}

We note that, for $s=0$, $s=1$, we have true orthogonality:
$$
\| f_1 + f_2 \|^2_{\dot{H}^s(\R^d)}
=
\| f_1 \|^2_{\dot{H}^s(\R^d)}
+
\| f_2 \|^2_{\dot{H}^s(\R^d)}
\,,
$$
owing to the local nature of the corresponding norms.

\begin{proof}[Proof of Lemma~\ref{l:orto}]
It is enough to observe that
$$
\begin{aligned}
\int_{\R^d} \int_{\R^d} & \frac{| f_1(x) +f_2(x) - f_1(y) - f_2(y)|^2}{|x-y|^{d+2s}} \, dx \, dy \\
& \geq
\int_{\Omega_1} \int_{\Omega_1} \frac{| f_1(x)  - f_1(y) |^2}{|x-y|^{d+2s}} \, dx \, dy
+
\int_{\Omega_2} \int_{\Omega_2} \frac{| f_2(x)  - f_2(y) |^2}{|x-y|^{d+2s}} \, dx \, dy
\end{aligned}
$$
and to apply the previous lemma.
\end{proof}

\bibliography{ACMloss_reg}
\bibliographystyle{abbrv}

\end{document}